\documentclass[3p,times,discrete]{elsarticle}

\usepackage{ecrc}

\volume{00}

\firstpage{1}

\journalname{Discrete Mathematics}

\runauth{}

\jid{procs}

\jnltitlelogo{Discrete Mathematics}

\CopyrightLine{2011}{Published by Elsevier Ltd.}

\usepackage{amsmath,amsfonts,amssymb}

\usepackage[figuresright]{rotating}
\newtheorem{theorem}{Theorem}
\newtheorem{corollary}{Corollary}
\newtheorem{proposition}{Proposition}
\newtheorem{lemma}{Lemma}
\newenvironment{proof}{\noindent {\bf Proof:}\ }{\ \rule{1mm}{2mm}}
\newcommand{\supp}{\hbox{supp}}

\begin{document}

\begin{frontmatter}



\dochead{}

\title{Transitive nonpropelinear perfect codes}

\author[rvt]{I. Yu. Mogilnykh}
\ead{ivmog@math.nsc.ru}


\author[rvt]{F. I. Solov'eva}
\ead{sol@math.nsc.ru}

\address{I. Yu. Mogilnykh and F. I. Solov'eva are
with the Sobolev Institute of Mathematics and Novosibirsk State
University, Novosibirsk, Russia.}

\fntext[rvt]{The paper is accepted to publishing in Discrete
Mathematics. The first author was supported by the Grants RFBR
12-01-00448 and 13-01-00463. The work of the second author was
supported by the Grant RFBR 12-01-00631-a. Both authors are
supported by the Grant NSh-1939.2014.1 of President of Russia for
Leading Scientific Schools.}

\begin{abstract}
A code is called transitive if its automorphism group (the
isometry group) of the code acts transitively on its codewords. If
there is a subgroup of the automorphism group acting regularly on
the code, the code is called propelinear. Using Magma software
package we establish that among 201 equivalence classes of
transitive perfect codes of length 15 from \cite{ost} there is a
unique nonpropelinear code. We solve the existence problem for
transitive nonpropelinear perfect codes for any admissible length
$n$, $n\geq 15$. Moreover we prove that there are at least 5
pairwise nonequivalent such codes for any admissible length $n$,
$n\geq 255$.
\end{abstract}

\begin{keyword}
perfect code \sep Mollard code \sep transitive action \sep regular
action

\MSC 94B25



\end{keyword}

\end{frontmatter}



\section{Introduction}
We consider codes in the  Hamming space $F_2^n$ of binary vectors
of length $n$ equipped with the Hamming metric. The Hamming
distance $d(x,y)$ is the number of different coordinate positions
of vectors $x$ and $y$. The {\it code distance} of a code is the
minimal value for the Hamming distance of its distinct codewords.
The weight $wt(x)$ of a binary vector $x$ of length $n$ is defined
as the Hamming distance between $x$ and the all-zero vector $0^n$.
With a vector $x$ we associate the collection of nonzero
coordinates which we denote as supp$(x)$. A collection $C$ of
binary vectors of length $n$ is called a {\it perfect} (1-perfect)
code if any binary vector is at distance 1 from exactly one
codeword of $C$.

Let $x$ be a binary vector, $\pi$ be a permutation of the
coordinate positions of $x$. Consider the transformation $(x,\pi)$
acting on a binary vector $y$ by the following rule:
 $$(x,\pi)(y)=x+\pi(y),$$
where $\pi(y)=(y_{\pi^{-1}(1)},\ldots,y_{\pi^{-1}(n)})$.
 The composition of two
automorphisms $(x,\pi)$, $(y,\pi')$ is defined as follows
$$(x,\pi)\cdot(y,\pi')=(x+\pi(y),\pi\circ\pi'),$$

where $\circ$ is a composition of permutations $\pi$ and $\pi'$.

The automorphism group of the Hamming space $F_2^n$ is defined as
$\mathrm{Aut}(F_2^n)=$  $\{(x,\pi): x \in C, \pi\in S_n, \,\,
x+\pi(F_2^n)=F_2^n\}$ with the operation composition, here $S_n$
denotes the group of symmetries of order $n$.

 The {\it automorphism group} $\mathrm{Aut}(C)$ of a code $C$ is a collection of all transformations $(x,\pi)$ fixing $C$
 setwise. In sequel for the sake of simplicity we require the all-zero vector to be always in
the code. Then we have the following representation for
$\mathrm{Aut}(C)$: $\{(x,\pi), x \in C, \pi\in S_n, \,\,
x+\pi(C)=C\}$.

 A code $C$ is called {\it
transitive} if there is a subgroup $H$ of $\mathrm{Aut}(C)$ acting
transitively on the codewords of $C$. If we additionally require
that for a pair of distinct codewords $x$ and $y$, there is a
unique element $h$ of $H$ such that $h(x)=y$, then $H$ acting on
$C$ is called a {\it regular group} \cite{PhelpsRifa} (sometimes
called sharply-transitive) and the code $C$ is called {\it
propelinear} (for the original definition see \cite{Rifa}). In
this case the order of $H$ is equal to the size of $C$. If $H$ is
acting regularly on $C$, we can establish a one-to-one
correspondence between the codewords of $C$ and the elements of
$H$ settled by the rule $x\rightarrow h_x$, where $h_x$ is the
automorphism sending a certain prefixed codeword (in sequel the
all-zero vector) to $x$.
 Each regular subgroup $H<\mathrm{\mathrm{Aut}}(C)$ naturally induces a
group operation on the codewords of $C$ in the following way:
$x*y:=h_x(y)$, such that the codewords of $C$ form a group with
respect to the operation $*$, isomorphic to $H$: $(C,*) \cong H$.
The group is called a {\it propelinear structure} on $C$. The
notion of propelinearity is important in algebraic and
combinatorial coding theory because it provides a general view on
linear and additive codes \cite{BR}.

Two codes $C$ and $D$ are called {\it equivalent} if there is an
automorphism $\phi$ of the Hamming space such that $\phi(C)=D$.
Equivalence or permutational equivalence (i.e. when
$\phi=(0^n,\pi)$) reduction is also often considered in problems
of classification and existence  for codes. Throughout in what
follows we consider all codes to contain all-zero vectors. In this
 case for the class of transitive codes the notions of equivalence and permutational
equivalence coincide.

For length 7, there is just one equivalence class of perfect
codes, containing the Hamming code (a unique linear perfect code).
A significant empirical boost of the study of perfect codes theory
was made by \"{O}sterg{\aa}rd and Pottonen who enumerated all
equivalence classes of perfect codes of length 15 (see \cite{ost}
for the database of the codes). In \cite{ost2} it was established
that 201 of 5983 such classes are transitive.

Papers of Avgustinovich \cite{Avg}, \cite{Avg2} provide a graphic
point of view on the problem of equivalence of perfect codes by
showing that two codes with isomorphic minimum distance graphs are
equivalent. In light of this result, transitive and propelinear
perfect codes have transitive and Cayley minimum distance graphs
respectively \cite{PhelpsRifa}. This fact relates the topic of our
work to a well known problem of the existence of transitive
non-Cayley graphs.

Note the definitions imply that a propelinear code is necessarily
transitive, however  both topics were studied by several different
authors and were developed somewhat independently.

In \cite{S2004}, \cite{S2005} Solov'eva showed that the
application of the Vasil'ev, Plotkin and Mollard constructions to
 transitive codes gives transitive codes. An analogous fact
for propelinearity was shown for Vasil'ev codes earlier in
\cite{RPB} and later in \cite{BorgesMogilnykhRifaSoloveva} for the
Plotkin and Mollard constructions. Studying 1-step switching class
of the Hamming code, Malyugin \cite{Mal} found several transitive
perfect codes of length 15 (they were shown to be propelinear
later in \cite{BorgesMogilnykhRifaSoloveva}).

 The first nonadditive propelinear codes of different ranks were found in \cite{BorgesMogilnykhRifaSoloveva}.
An asymptotically exponential of length class of transitive
extended perfect codes constructed in \cite{Pot} were shown to be
propelinear in \cite{BorgesMogilnykhRifaSoloveva2}. In \cite{KrotovPotapov} Potapov and Krotov   utilized quadratic functions in the Vasil'ev construction to obtain propelinear perfect codes.
 Because these codes are only of small rank the question of
  the existence of a big (e.g. exponential of $n$) class of large rank propelinear perfect codes is still open.

The first transitive code that was shown not to have a propelinear
representation was the well known Best code of length $10$ and
code distance $4$ \cite{BorgesMogilnykhRifaSoloveva}. In the same
work the question of the existence of transitive nonpropelinear
perfect code was proposed.

The aim of this work is to separate the classes of transitive and
propelinear perfect codes for any admissible length $n$. Using
Magma software package, we found that only one of 201 transitive
perfect codes of length 15 is nonpropelinear. The code is
characterized in the class of transitive codes of length 15 by a
unique property of having no triples from the kernel. The
extension of this  code by  parity check gives a propelinear code.
Since adding parity check preserves propelinearity of a code, we
conclude that all extended perfect codes of length 16 are
propelinear. In the paper we present the solution of the problem
of the existence of transitive nonpropelinear perfect code for any
admissible length $n$, $n\geq 15$. Moreover we show that there
exist nonequivalent transitive nonpropelinear perfect code for any
admissible length more than $127$.

The current paper is organized as follows. Definitions and basic
theoretical facts are given in the second section.  The case
$n=15$ is considered in Section 3, where we give some information
on the transitive nonpropelinear code and describe the way the
search was carried out. A treatment of nonpropelinearity of the
transitive nonpropelinear code $C$ of length 15 is in Section 4 as
well as a sufficient condition for an extension of this property
for the Mollard codes $M(C,D)$ for the appropriately chosen code
$D$. The condition is essentially a restriction on the orbits of
action of the symmetry groups of the Mollard codes. The condition
holds if the Mollard code has certain metrical properties which we
formulate by means of  a numerical invariant $\mu_i(C)$ (the
number of the triples from the kernel of the code incident to
coordinate $i$). The main result of the paper is given in Section
5.

\section{Preliminaries and notations}

\subsection{Mollard code}
First give a representation for the Mollard construction
\cite{Mol}. Let $C$ and $D$ be two codes of lengths $t$ and $m$.
Consider the coordinate positions of the Mollard code $M(C,D)$ of
length $tm+t+m$ to be pairs $(i,j)$ from the set
$\{0,\ldots,t\}\times \{0,\ldots,m\}\setminus (0,0)$.

Let $f$ be an arbitrary function from $C$ to the set of binary
vectors $Z_2^m$ of length $m$ and let $p_1(z)$ and $p_2(z)$ be the
generalized parity check functions:

$$p_1(z)=(\sum_{j=0}^{m}z_{1,j},\ldots,\sum_{j=0}^{m}z_{t,j}),$$
 $$p_2(z)=(\sum_{i=0}^{t}z_{i,1},\ldots,\sum_{i=0}^{t}z_{i,m}).$$

 The code $M(C,D)=\{z\in Z_2^{tm+t+m}: p_1(z)\in C, p_2(z)\in
 f(p_1(z))+D\}$ is called
 a Mollard code.
In the case when $C$ and $D$ are perfect, the code $M(C,D)$ is
perfect. Throughout the paper we consider the
 case when $f$ satisfies $f(p_1(z))=0^m$ for any  $p_1(z)\in C$.

A {\it Steiner triple system} is a set of $n$ points together with
 a collection of blocks (subsets) of size 3
of points, such that any
 unordered pair of distinct points is exactly in one block. Further we put the triples of $STS$ into round brackets to distinguish them with the supports of vectors.
  The set of codewords of weight 3 in a perfect code $C$, that contains the all-zero
 codeword defines a Steiner triple system, which we denote ${\mathrm{STS}}(C)$.

By the Mollard construction it is easy to see that
$\mathrm{STS}(M(C,D))$ can be defined as
$$\mathrm{STS}(M(C,D))=\bigcup_{k,p\in \{0,3\}} T_{kp}, \,  \mathrm{where}$$
$$T_{00}=\{((r,0),(r,s), (0,s)): r\in \{1,\ldots,t\}, s\in \{1,\ldots,m\}\};$$
$$T_{33}=\{((r,s), (r',s'), (r'',s'')): (r, r', r'')\in \mathrm{STS}(C), (s, s', s'') \in \mathrm{STS}(D)\};$$ 
 $$T_{30}=\{ ((r,0),(r',s), (r'',s)): (r, r', r'')\in \mathrm{STS}(C), s\in \{0,1,\ldots,m\}\};$$
$$T_{03}=\{((r,s), (r,s'), (0,s'')): (s, s', s'')\in \mathrm{STS}(D), r\in \{0,1,\ldots,t\}\}.$$

 Denote
by  ${x}^{1}$  (${y}^{2}$  respectively)
 the codeword in  $M(C,D)$ such that   $({x}^{1}_{0,1},\ldots,$
${x}^1_{0,m})=x\in C$
($({y}^{2}_{1,0},\ldots,{y}^{2}_{t,0})$$=y\in D$ respectively)
with zeros in all positions from
${\{0,\ldots,t\}\times\{1,\ldots,m\}}$ ($\{1,\ldots,t\}$
$\times\{0,\ldots,m\}$ respectively). Note that $M(C,D)$ contains
the codes $C$ and $D$ as the subcodes $M(C,0^m)=\{{x}^{1}: x \in
C\}$ and $M(0^t,D)=\{{y}^{2}: y \in D\}$ respectively.

Recall that {\it the dual} $C^{\perp}$ of a code $C$ is a
collection of all binary vectors $x$ such that
$\sum_{i=1}^{n}x_ic_i=0$$\mbox{(mod 2)}$ for any codeword $c$ of
$C$. Denote by $I(C)$ the following set:
$$I(C)=\{i:x_i=0 \,\,
\mathrm{for \,\, all} \,\, x \in C^{\perp}\}.$$

It is easy to see that
\begin{equation}\label{zeroMollard}I(M(C,D))=\{(r,s):r \in I(C)\cup 0, s\in I(D)\cup 0\}\setminus (0,0).\end{equation}

The  {\it rank} of a code is defined to be  the dimension of its
linear span and the {\it  kernel} of the code to be the subspace
$\mathrm{Ker}(C)=\{x\in C: x+C=C\}$.
  The rank and kernel are important code invariants. Due to its structure, the Mollard code preserves many
  properties and characteristics
  of the initial codes, in particular, we have the iterative formulas for the size of kernel and rank:
\begin{equation}\label{Mollard_Ker}Dim(\mathrm{Ker}(M(C,D)))=Dim(\mathrm{Ker}(C))+Dim(\mathrm{Ker}(D))+tm,\end{equation}
\begin{equation}\label{Mollard_Rank}Rank(M(C,D))=Rank(C)+Rank(D)+tm.\end{equation}
The previous formula was used in \cite{GMS} in solving the rank
problem for propelinear perfect codes.

\subsection{Automorphism group of a perfect code}

The {\it symmetry group} $\mathrm{Sym}(C)$ of a code $C$ of length
$n$ (sometimes being called permutational automorphism group or
full automorphism group \cite{MS}) is the collection of
permutations on $n$ elements
 with the operation composition, preserving the code
setwise:
$$\mathrm{Sym}(C)=\{\pi \in S_n: \pi(C)=C \}.$$

The {\it group of rotations}, see \cite{AHS},
\cite{BorgesMogilnykhRifaSoloveva}, ${\mathcal R}(C)$, consists of
all permutations
 with the operation composition,
 that could be embedded into the permutational
part of an automorphism of $C$:

$${\mathcal R}(C)=\{\pi:\textrm{there  exists} \,\,\, x \in C\,\,\, \textrm{such  that} \,\,\, (x,\pi)\in \mathrm{Aut}(C)\}.$$

Obviously, the symmetry group is a subgroup of the group of
rotations. On the other hand, ${\mathcal R}(C)$ stabilizes the
dual of the code and its kernel \cite{PhelpsRifa},
\cite{BorgesMogilnykhRifaSoloveva}, so we have
\begin{equation} \label{SymKer}
\mathrm{Sym}(C)\leq {\mathcal R}(C)\leq
\mathrm{Sym}(\mathrm{Ker}(C)),
\end{equation}

\begin{equation} \label{SymRot}
{\mathcal R}(C)\leq \mathrm{Sym}(C^{\perp}).
\end{equation}

In section \ref{mu_sec} we make use of the following known
statement, which is a straightforward consequence of
(\ref{SymRot}).
\begin{lemma}\label{zero} If $I(C)$ is the collection of zero coordinate positions for the dual of $C$, then ${\mathcal
R}(C)$ stabilizes $I(C)$ setwise. \end{lemma}

Finally, the constant weight subcode of the code is stabilized by
symmetries of the code, so in case of weight 3 we have
\begin{equation} \label{SymSTSC}
\mathrm{Sym}(C) \leq \mathrm{Aut}(\mathrm{STS}(C)),
\end{equation}
here and below by $Aut(STS(C))$ we mean the automorphism group of
Steiner triple systems, i.e. the symmetry group of $STS(C)$
treated as a binary code.

Denote by  ${\mathcal R}_{x}(C)$ the set of elements of ${\mathcal
R}(C)$ associated with a codeword $x$ of $C$:
 $${\mathcal R}_x(C)=\{\pi: (x,\pi)\in \mathrm{Aut}(C)\}.$$

It is easy to see that the introduced sets are exactly cosets of
${\mathcal R}(C)$ by $\mathrm{Sym}(C)$ (see
\cite{BorgesMogilnykhRifaSoloveva}):
\begin{equation} \label{Gx}
 {\mathcal R}_x(C)=\pi \mathrm{Sym}(C),
\end{equation}
for any $\pi \in {\mathcal R}_x(C).$

Now consider the Mollard code $M(C,D)$. For a permutation $\pi$ on
the coordinate positions of the code $C$, denote by ${\mathcal
D}_1(\pi)$
  a permutation
 on the coordinates of $M(C,D)$: ${\mathcal
 D}_1(\pi)(r,s)=(\pi(r),s)$ for $r\ne 0$ and ${\mathcal
 D}_1(\pi)(0,s)=(0,s)$ (see \cite{S2005},
\cite{BorgesMogilnykhRifaSoloveva}). For a permutation $\pi$ on
the coordinate
 positions of $D$, define ${\mathcal
 D}_2(\pi)(r,s)=(r,\pi(s))$ for $s\ne 0$ and ${\mathcal
 D}_2(\pi)(r,0)=(r,0)$.

 If $(x,\pi_x)$ and $(y,\pi_y)$
  are automorphisms of $C$ and $D$ respectively, then there is an automorphism $(z,{\mathcal D}_1(\pi_x)
 {\mathcal D}_2(\pi_y))$ of $M(C,D)$ for any $z$ such that $p_1(z)=x$, $p_2(z)=y$, see \cite{S2005}.
 In particular this fact shows that the Mollard construction preserves transitivity. So, we have the following
 facts:

\begin{lemma}\label{RotMollard}
Let $C$ and $D$ be perfect codes, $z$ be a codeword of the Mollard
code $M(C,D)$. Then \begin{enumerate}
    \item ${\mathcal R}_z(M(C,D))={\mathcal D}_1({\mathcal R}_{p_1(z)}(C)){\mathcal D}_2({\mathcal R}_{p_2(z)}(D))\mathrm{Sym}(M(C,D)),$

\item \cite{S2005} ${\mathcal D}_1(\mathrm{Sym}(C))\leq  \mathrm{Sym}(M(C,D)), \,\,\, {\mathcal D}_2(\mathrm{Sym}(D))\leq
\mathrm{Sym}(M(C,D)).$
\end{enumerate}

\end{lemma}

\begin{lemma}\cite{S2005}
If $C$ and $D$ are transitive codes, then $M(C,D)$ is transitive.
\end{lemma}




\section{Propelinear perfect codes of length 15}\label{invsec}

We give the original definition of a propelinear code. A code is
called {\it propelinear} \cite{Rifa} if

\begin{itemize}

\item[(i)]
 each $x\in C$ could be assigned
 a coordinate permutation $\pi_x\in {\mathcal R}_x(C)$;

\item[(ii)] the attached permutations satisfy:

if $(x,\pi_x)(y)=z$, then $\pi_z$ is a composition of $\pi_x$ and
$\pi_y$, i.~e.  $\pi_z=\pi_x\circ \pi_y$, for any $y\in C$.
\end{itemize}

The property $(i)$ is equivalent to transitivity of the group
generated by the set of transformations $\{(x,\pi_{x}): x \in
C\}$. While the addition of the property (ii) amounts to the fact
that the set of transformations $\{(x,\pi_{x}): x \in C\}$ forms a
group itself \cite{PhelpsRifa}.

A {\it normalized propelinear code}
\cite{BorgesMogilnykhRifaSoloveva} is defined by additionally
requiring that the number of assigned permutations is minimal:

\begin{itemize}

\item[(iii)] $|\{\pi_x:x \in C\}|=|C|/|\mathrm{Ker}(C)|$.

\end{itemize}

There are 201 perfect transitive codes of length 15, see
\cite{ost}. See also \cite{gussol} for more information on these
codes, e.g. the sizes of ranks, kernels, etc. Using Magma
\cite{MagmaCitation}, we studied the set of transitive perfect
codes of length 15.
 For a given code, we checked if it is normalized propelinear \cite{BorgesMogilnykhRifaSoloveva}, which is a relatively quick procedure.
 It turned out that there is just one transitive code (number 4918 in
the database \cite{ost}), which is not normalized propelinear,
while the other 200 codes admit a normalized propelinear structure
(and therefore are propelinear). Moreover, a further search showed
that there is no regular subgroup of the automorphism group of the
code number 4918, so the code is nonpropelinear. This code of
length 15 has a characteristic property of having the minimum
distance of its kernel equaled 4, while remaining 200 transitive
codes of length 15 have this parameter equal to 3.

\begin{proposition}
There is a unique transitive nonpropelinear perfect code of length
15.
\end{proposition}

The Steiner triple system of the transitive nonpropelinear code
has rank 14 and therefore has a unique subsystem of order 7, see
\cite{due}, on the coordinate positions $\{1,2,3,4,6,7,8\}$. Note
that $\mathrm{Aut}(\mathrm{STS}(C))$ fixes the coordinates of a
unique subsystem of order 7 setwise.

\begin{table}[h!]\caption{{\normalsize Steiner triple system of the transitive nonpropelinear code of length 15}}\label{Table T1}
\label{table STS}
\begin{center}
\begin{tabular}{|c|c|c|c|c|}
\hline
Sub$\mathrm{STS}$&&&&\\
\hline
  ( 1, 3, 7 ) & ( 6, 11, 12 ) & ( 1, 5, 9 ) & ( 5, 7, 13 ) & ( 8, 10, 15 ) \\
  ( 3, 6, 8 )&  (1, 11, 14 ) & ( 6, 10, 14 ) & ( 7, 9, 10 ) & ( 2, 14, 15 ) \\
  ( 1, 4, 8 ) & ( 4, 10, 11 ) & ( 5, 8, 14 ) & ( 1, 12, 15 ) & ( 4, 5, 12 ) \\
  ( 1, 2, 6 ) & ( 3, 10, 12 ) & ( 6, 9, 13 ) & ( 5, 6, 15 ) & ( 8, 12, 13 ) \\
  ( 2, 3, 4 ) & ( 1, 10, 13) &  ( 3, 13, 15) & ( 7, 12, 14 ) & ( 3, 9, 14 ) \\
  ( 4, 6, 7 ) & ( 2, 9, 12 ) & ( 2, 5, 10 ) & ( 4, 13, 14 ) & ( 2, 11, 13 ) \\
  ( 2, 7, 8 ) & ( 8, 9, 11 ) & ( 3, 5, 11 ) & ( 7, 11, 15 ) & ( 4, 9, 15 ) \\
  \hline
 \end{tabular}
\end{center}
\end{table}

The following permutations, together with the identity permutation
form $\mathrm{Aut}(\mathrm{STS}(C))$, which coincides with
$\mathrm{Sym}(C)$:
$$(5, 15)(9, 12)(10, 14)(11, 13),$$
    $$(5, 10)(9, 13)(11, 12)(14, 15),$$
    $$(5, 14)(9, 11)(10, 15)(12, 13).$$

We define the transitive nonpropelinear perfect code of length 15
using its cosets of kernel, see Tables \ref{Table T21} and
\ref{Table T2}.
\begin{table}[h!]\caption{\normalsize Supports of the base of $\mathrm{Ker}(C)$}\label{Table T21}\begin{center}\begin{tabular}{|c|c|} \hline

  \{ 9, 11, 12, 13 \}& \{ 4, 6, 7, 8, 11, 12, 14, 15 \} \\
   \{ 5, 12, 13, 14 \} & \{ 2, 3, 6, 7, 11, 13, 14, 15 \} \\
  \{10, 12, 13, 15 \} &  \{ 1, 6, 7, 12, 13 \} \\ \hline
\end{tabular}
\end{center}
\end{table}
\noindent

\begin{table}[h!]\caption{\normalsize Supports of the cosets of $\mathrm{Ker}(C)$}\label{Table T2}\begin{center}
\begin{tabular}{|c|c|c|c|}
\hline
$\{ 1, 12, 15 \}$&$\{ 4, 6, 7 \}$&$\{ 2, 4, 7, 10, 12 \}$&$\{ 5, 6, 8, 9 \}$\\
$\{ 4, 9, 15 \}$& $\{ 4, 6, 14, 15 \}$&$\{ 2, 11, 13 \}$&$\{ 1, 4, 8 \}$\\
$\{ 3, 9, 14 \}$& $\{ 5, 7, 13 \}$& $\{ 1, 2, 6 \}$&$\{ 3, 5, 7, 12 \}$\\
$\{ 4, 13, 14 \}$& $\{ 1, 3, 7 \}$& $\{ 2, 6, 8, 13, 15 \}$& $\{ 2, 4, 9, 13 \}$\\
$\{ 2, 8, 10, 11 \}$&$\{ 8, 9, 11 \}$& $\{ 3, 6, 8 \}$&$\{ 6, 9, 13 \}$\\
$\{ 5, 7, 8, 15 \}$& $\{ 1, 5, 9 \}$&$\{ 2, 5, 6, 13 \}$&$\{ 6, 10, 14 \}$\\
$\{ 2, 5, 10 \}$&$\{ 3, 4, 5, 13 \}$&$\{ 3, 13, 15 \}$&$\{ 7, 11, 15 \}$\\
$\{ 2, 7, 8 \}$&$\{ 6, 8, 12, 15 \}$&$\{ 3, 5, 8, 10 \}$&\\
 \hline
\end{tabular}
\end{center}
\end{table}

In Table \ref{TableTR15} we give
 some
parameters of the considered codes in this paper. The codes have
special properties, which we will use later in Section 5.
 Let $C$ be a code of length $n$, then for any $i\in
\{1,\ldots,n\}$ define $\mu_i(C)$ to be the number of triples from
$\mathrm{Ker}(C)$ that contain $i$. From (\ref{SymKer}) and
(\ref{SymSTSC}) we see that $\mu_i(C)\neq \mu_j(C)$ implies that
the coordinates $i$ and $j$ are in different orbits of the group
action of $\mathrm{Sym}(C)$ on the coordinate positions
$\{1,\ldots,n\}$. In Table \ref{TableTR15} and further $\mu(C)$ is
 the multiset
 collection of $\mu_i(C)$  denoted by $\mu_{k_1}^{i_1}\mu_{k_2}^{i_2}\ldots\mu_{k_p}^{i_p},$ $p\leq n$ (here
the integer $\mu_{k_l}$ appears $i_l$, $i_l\ne 0$ times, $1\leq
l\leq p$) for any coordinate $i$ of $C$.
\begin{table}[h!]\caption{\normalsize Some transitive perfect codes of length 15}
\label{TableTR15}
\begin{center}
\noindent\begin{tabular}{|c|c|c|c|c|c|c|}
  \hline
    {\small Code } &&&&&& \\
  {\small number} &{\small Rank(C)}&{\small Dim($\mathrm{Ker}(C)$)}&{\small $|\mathrm{Sym}(C)|$}&{\small $\mu(C)$}&{\small $|\mathrm{Aut}(\mathrm{STS}(C))|$}&{\small Rank($\mathrm{STS}(C)$)} \\
  {\small in  \cite{ost}} &&&&&&   \\
   \hline
  51 & 13 & 7&8 & $1^{13}3^15^1$& 8 & 13   \\
  694 & 13 & 8&32 & $1^{8}3^55^2$& 32 & 13   \\
  724 & 13 & 8&32 & $1^{13}3^15^1$& 96 & 13   \\
  771 & 13 & 8&96 & $1^{12}3^3$& 288 & 13   \\
  4918 & 14 & 6 &4 & 0$^{15}$& 4 & 14  \\
  \hline
 \end{tabular}
 \end{center}
 \end{table}


\section{Transitive nonpropelinear perfect codes}

We say that a codeword $x$ of $C$ has {\it the incorrect inverse},
if any element of ${\mathcal R}_x(C)$ is of order more than 2 and
stabilizes $\supp(x)$.

\begin{proposition}\label{IncorrectInverse}
A code $C$ containing a codeword $x$ with the incorrect
 inverse is not propelinear.
\end{proposition}
\begin{proof}
Suppose $H$ is a regular subgroup of the automorphism group of
$C$. Let $h_x=(x,\pi_{x})\in H$ be the automorphism
attached to $x$, i.e.  $h_x$ maps $0$ into $x$. Then
$h_x^{-1}=(\pi_x^{-1}(x), \pi_x^{-1}) \in H$ maps the all-zero
codeword to $\pi_x^{-1}(x)$. Because $H$ is a regular group, there
is a unique element of $H$ sending $0$ to $x$. However we have
that $\pi_x^{-1}(x)=x$ and therefore the automorphisms $h_x$ and
$h_{x}^{-1}$ must be equal, because they both map the all-zero
codeword to $x$. So we get that $\pi_x^{2}$ is the identity
permutation for some $\pi_x\in {\mathcal R}_x(C)$, which
contradicts the fact that $x$ is a codeword
 with the incorrect inverse.
\end{proof}

\begin{corollary}\label{SymStab}
If $C$ is a code containing a codeword $x$ with the incorrect
inverse, then $\mathrm{Sym}(C)$ is of even order and stabilizes
supp$(x)$ setwise.
\end{corollary}
\begin{proof}
From the proof of Proposition \ref{IncorrectInverse} we have that
for any $\pi_x\in {\mathcal R}_x(C)$ the transformation
$(\pi_x^{-1}(x),\pi_x^{-1})=(x,\pi_x^{-1})$ is the automorphism of
$C$. So, the set ${\mathcal R}_x(C)$ is closed under inversion.
This fact combined with the fact that the square of any element of
${\mathcal R}_x(C)$ is not the identity, implies that $|{\mathcal
R}_x(C)|$ is even. Moreover since ${\mathcal R}_x(C)$ is a coset
of $\mathrm{Sym}(C)$, the group generated by the elements of
${\mathcal R}_x(C)$ contains $\mathrm{Sym}(C)$ and therefore
$\mathrm{Sym}(C)$ inherits the property of stabilizing supp$(x)$
setwise from ${\mathcal R}_x(C)$, because ${\mathcal R}_x(C)$ is
closed under inversion.
 \end{proof}

We make use of the following empirical fact, established by Magma
software package.
\begin{proposition}\label{TNP15}
The code number 4918 in classification of \cite{ost} is transitive
and contains a codeword $x$, supp$(x)=\{2,3,4\}$ with the
incorrect inverse.
\end{proposition}

\begin{lemma}\label{act}
Let $C$ and $D$ be perfect codes of lengths $t$ and $m$
respectively, $\sigma$ be a permutation from
$\mathrm{Sym}(M(C,D))$ that stabilizes the subcode $M(0^t,D)$
setwise. Then there is an element $\pi$ in $\mathrm{Sym}(C)$ such
that for any coordinate $(r,s)$, $r\in \{1,\ldots,t\}, \, s\in
\{0,\ldots,m\}$ it holds $\sigma(r,s)=(\pi(r),s')$ for some $s'\in
\{0,\ldots,m\}$.
\end{lemma}
\begin{proof}
Let $\bar{s}$ be in $\{1,\ldots,m\}$. Consider the triple $((r,0),
(r,\bar{s}),(0,\bar{s}))$ in $M(C,D)$ for some $r\in
\{1,\ldots,t\}$. Let $\sigma(0,\bar{s})$ be $(0,s''')$, then
$\sigma(r,\bar{s})=(r',s')$ and $\sigma(r,0)=(r'',s'')$ for some
$s', s''$ and nonzero $r',r''\neq0$. From the description of the
triple set in the Mollard code we see that $r''=r'$, i.e.
$\sigma((r,0),(r,\bar{s}),(0,\bar{s}))$ is in $T_{03}$ or
$T_{00}$.
 In other words, $\sigma$ acts
as a permutation $\pi\in S_t$ on the first coordinate of
coordinates-pairs $(r,s)$: $\sigma(r,s)=(\pi(r),s')$, for nonzero
$r$ and any $s\in \{0,\ldots,m\}$.
 The permutation $\pi$ belongs to $\mathrm{Sym}(C)$ since $\sigma$
should act as a permutation from $\mathrm{Sym}(C)$ on the first
coordinate of the subcode $M(C,0^m)$:
  for all $ x\in C \,\, \mbox{we have}
\,\,\, p_1(\sigma({x}^{1}))=p_1({\pi(x)}^1)=\pi(x)\in C \mbox{ iff
} \pi \in \mathrm{Sym}(C).$\
\end{proof}

The following statement gives a sufficient condition for a Mollard
code to preserve the incorrect inversion property of one of the
codes in terms of the code symmetries.
\begin{lemma}\label{Necesscond} Let $C$ be a perfect code, $x$ be a codeword of $C$ with the incorrect inverse, $D$ be a perfect code of length $m$ such that
\begin{equation}\label{prop1}\mbox{for any }\sigma \in \mathrm{Sym}(M(C,D)) \mbox{ we have }
\sigma(M(0^t,D))=M(0^t,D)
\end{equation}

and
 \begin{equation}\label{prop2}\mbox{for any } \sigma \in \mathrm{Sym}(M(C,D)) \mbox{ we have  } \sigma({x}^{1})\in M(C,0^m).
\end{equation}
  Then ${x}^{1}$ is a
codeword with the incorrect inverse in $M(C,D)$.
\end{lemma}
\begin{proof}
By Lemma \ref{RotMollard} we have
$${\mathcal R}_{{x}^{1}}(M(C,D))={\mathcal D}_1({\mathcal R}_x(C)){\mathcal D}_2({\mathcal R}_{0^{m}}(D))\mathrm{Sym}(M(C,D))=$$ $${\mathcal D}_1({\mathcal R}_x(C)){\mathcal D}_2(\mathrm{Sym}(D))\mathrm{Sym}(M(C,D))=
{\mathcal D}_1({\mathcal R}_x(C))\mathrm{Sym}(M(C,D)).$$

By Lemma \ref{act} a permutation $\sigma\in \mathrm{Sym}(M(C,D))$
sends any element $(r,0)$ of $\supp({x}^{1})=\{(r,0):r \in
\supp(x)\}$ to $(\pi(r),s)$ for some $\pi \in \mathrm{Sym}(C)$ and
$s\in \{0,\ldots,m\}$. By Corollary \ref{SymStab} the permutation
$\pi$ stabilizes the
 codeword $x\in C$ with the incorrect inverse, so  $\pi(r)\in supp(x)$ and by
the condition (\ref{prop2}) we have that $s$ is 0, i.e. ${x}^{1}$
is stabilized by $\mathrm{Sym}(M(C,D))$. Now if $\pi'\in {\mathcal
R}_x(C)$, then ${\mathcal
D}_1(\pi')({x}^{1})={\pi'(x)}^{1}={x}^{1}$ because $x$ is a
codeword with the incorrect inverse in $C$. So ${\mathcal
D}_1({\mathcal R}_x(C))\mathrm{Sym}(M(C,D))$ stabilizes ${x}^{1}$.

By Lemma \ref{act}, we see that the action of an element  from
${\mathcal D}_1({\mathcal R}_x(C))\mathrm{Sym}(M(C,D))$ on the
first coordinate of the coordinates-pairs of $M(C,D)$ is realized
by an element from ${\mathcal R}_x(C)$, so for any $\sigma\in
{\mathcal D}_1({\mathcal R}_x(C))\mathrm{Sym}(M(C,D))$ there is
$\pi$ such that
$$\sigma^2(r,s)=(\pi^2(r),s'), \mbox{ for all } r\in \{1,\ldots,t\}, s\in\{0,\ldots,m\}.$$

 From the equality above we see that the order of
any element from ${\mathcal R}_{{x}^{1}}(M(C,D))$ is not less then
that of any element of ${\mathcal R}_x(C)$ and therefore is more
than 2. Thus the codeword $x^{1}$ of $M(C,D)$ has the incorrect
inverse.
\end{proof}

\section{Transitive nonpropelinear Mollard codes}\label{mu_sec}

 In this section we use of the parameters $\mu_r(C)$ and $\mu(C)$ in constructing transitive nonpropelinear Mollard codes. We utilize the iterative structure of
$\mathrm{STS}(M(C,D))$ and obtain formulas for
$\mu_{(r,s)}(M(C,D))$ for the Mollard code $M(C,D)$ from
$\mu_r(C)$ and $\mu_s(D)$. Further we derive a metric version of
Lemma \ref{Necesscond} and construct 5 infinite series of perfect
transitive nonpropelinear Mollard codes. Recall that
$\mathrm{STS}(M(C,D))$ could be presented as the union of the
following sets:
$$T_{00}=\{((r,0),(r,s), (0,s)): r\in \{1,\ldots,t\}, s\in \{1,\ldots,m\}\};$$
$$T_{33}=\{((r,s), (r',s'), (r'',s'')): (r, r', r'')\in \mathrm{STS}(C), (s, s', s'') \in \mathrm{STS}(D)\};$$ 
 $$T_{30}=\{ ((r,0),(r',s), (r'',s)): \{r, r', r''\}\in \mathrm{STS}(C), s\in \{0,\ldots,m\}\};$$
$$T_{03}=\{((r,s), (r,s'), (0,s'')): \{s, s', s''\}\in \mathrm{STS}(D), r\in \{0,\ldots,t\}\}.$$

\begin{lemma}\label{mu}
Let $M(C,D)$ be a Mollard code obtained from
  perfect codes $C$ and $D$ of length $t$ and
$m$ respectively. Then
\begin{enumerate}
    \item  $\mu_{(r,0)}(M(C,D))=\mu_{r}(C)(m+1)+m;$
\item $\mu_{(0,s)}(M(C,D))=\mu_{s}(D)(t+1)+t;$
\item $\mu_{(r,s)}(M(C,D))=1+2(\mu_{s}(D)+\mu_{r}(C)+\mu_{r}(C)\mu_{s}(D)).$
\end{enumerate}
\end{lemma}
\begin{proof}
First of all we note that a triple $\Delta \in
\mathrm{Ker}(M(C,D))$ iff $p_1(\Delta)\in \mathrm{Ker}(C)$ and
$p_2(\Delta)\in \mathrm{Ker}(D)$.

1. By the criteria above, since $T_{00}\subset \mathrm{Ker}(C)$, a
coordinate position $(r,0)$ is contained in $m$ kernel triples
$((r,0),(r,s),(0,s)),$ \, $s=1,\ldots,m$ from $T_{00}$. \, Also
$(r,0)$ is in $(m+1)\mu_{r}(C)$ triples $((r,0),(r',s'),(r'',s'))$
from $\mathrm{Ker}(M(C,D))\cap T_{03}$, where $(r, r', r'')\in
\mathrm{Ker}(C)$, $s'=0,\ldots,m$.

2. The proof is analogous to that of the first statement.

3.  The coordinate $(r,s)$ is in just one kernel triple from
$T_{00}$ which is $((r,0),(r,s), (0,s))$.

 Moreover, this coordinate is contained in $2\mu_r(C)$ and $2\mu_s(D)$ triples from $T_{30}\cap Ker(M(C,D))$ and $T_{03}\cap \mathrm{Ker}(M(C,D))$
 respectively that are equal to
  the sets $\{((r,s),(r',s),(r'',0)): (r, r', r'')\in \mathrm{Ker}(C), s=0,\ldots,m\}$ and $\{((r,s),(r,s'),$ $(0,s'')): (s, s',s'')\in \mathrm{Ker}(D), r=0,\ldots,t\}$ respectively.

Finally, there are $2\mu_{r}(C)\mu_{s}(D)$ triples

 $$\{(r,s), (r',s'), (r'',s''): (r,r',r'') \in \mathrm{Ker}(C),(s,s',s'') \in \mathrm{Ker}(D)\}$$
from $T_{33}\cap \mathrm{Ker}(M(C,D))$ containing $(r,s)$.

Summing up the number of triples for the above cases, we get the
desired value for $\mu_{(r,s)}(M(C,D))$.
\end{proof}

Recall that  $\mu(C)$ is defined above to be the multiset
collection of $\mu_i(C)$ for any coordinate $i$ of $C$. So
 $\mu(C)$ could be considered as a code invariant.
From Lemma \ref{mu} we immediately obtain

\begin{corollary}\label{mu_cor}
Let $\mu(C)\neq \mu(C')$ for codes $C$ and $C'$ containing $0^n$.
Then the codes $M(C,D)$ and $M(C',D)$ are nonequivalent.
\end{corollary}

Now we consider several conditions on the initial codes in order
for the Mollard construction to preserve the incorrect inversion
property.

\begin{theorem}\label{Th1} Let $C$ be a perfect code of length $t$ with a codeword $x$ with the incorrect inverse.

1. If we have

\begin{equation}\label{C_C1}
\supp(x)\subseteq I(C),
\end{equation}

\begin{equation}\label{C_C2}
\mu_r(C)<(t-1)/2 \, \mbox{ for any } \, r\in \{1,\ldots,t\},
\end{equation}
then ${x}^{1}$ is a codeword with the incorrect inverse in
$M(C,H)$.

2. If we have

\begin{equation}\label{C_C3}
\mu_r(C)=0 \, \mbox{ for any } \, r\in \{1,\ldots,t\},
\end{equation}

\begin{equation}\label{D_D1}
0<\mu_s(D)<\frac{m-1}{2} \, \mbox{ for any } \, s\in
\{1,\ldots,m\}, m\leq t,
\end{equation}
then ${x}^{1}$ is a codeword with the incorrect inverse in
$M(C,D)$.

3. If (\ref{C_C1}), (\ref{C_C3}) hold for $C$ and (\ref{D_D1})
holds for $D$, then ${x}^{1}$ is a codeword with the incorrect
inverse in $M(M(C,D),H)$ for any Hamming code $H$.
\end{theorem}
\begin{proof} We show that the conditions of Lemma \ref{Necesscond} are
satisfied.

Let us look at the values of $\mu$ for the coordinates of the
considered Mollard codes.

1.  By Lemma \ref{mu} we have the following relations for any
$r\in\{1,\ldots,t\}$ and $s\in\{1,\ldots,m\}$:
$$\mu_{(r,0)}(M(C,H))=\mu_{r}(C)(m+1)+m<(m+1)(t-1)/2+m=\frac{tm+t+m-1}{2},$$
$$\mu_{(0,s)}(M(C,H))=(m-1)(t+1)/2+t=\frac{tm+t+m-1}{2},$$
$$\mu_{(r,s)}(M(C,H))=1+2\mu(C,r)+m-1+(m-1)\mu(C,s)=m+\mu(C,s)(m+1)=$$
$$\mu_{(r,0)}(M(C,H)).$$

We see that the nonzero coordinates $\{(0,s): s=1,\ldots,m\}$ of
the subcode $M(0^t,H)$ are the only coordinates of $M(C,H)$ with
the maximum possible value for $\mu$. Therefore these coordinates
are stabilized by $\mathrm{Sym}(M(C,H))$ and we have the condition
(\ref{prop1}).

By (\ref{zeroMollard}) and (\ref{C_C1}) we obtain that
$\supp({x}^{1})=\{(r,0):r \in \supp(x)\}$ is a subset
$I((M(C,H)))=I(C)\times 0$ of zero coordinate positions for the
dual of $M(C,H)$. By Lemma \ref{zero} we have that
$\mathrm{Sym}(M(C,H))$ stabilizes the block $I((M(C,H)))$ setwise,
so the condition (\ref{prop2}) holds.

2. We have the following relations for any $r\in\{1,\ldots,t\}$
and $s\in\{1,\ldots,m\}$:
$$\mu_{(r,0)}(M(C,D))=\mu_{r}(C)(m+1)+m=m,$$
$$\mu_{(0,s)}(M(C,D))=\mu_s(D)(t+1)+t>m,$$
$$\mu_{(r,s)}(M(C,D))=1+2(\mu_r(C)+\mu_s(D)+\mu_r(C)\mu_s(D))=1+2\mu_s(D)<m.$$

The above implies that the sets  $\{(r,0): r=1,\ldots,t\}$,
$\{(0,s): s=1,\ldots,m\}$ and $\{(r,s):
s=1,\ldots,m,r=1,\ldots,t\}$ of coordinates of $M(C,D)$ are
stabilized by $\mathrm{Sym}(M(C,D))$, which implies that the
conditions (\ref{prop1}) and (\ref{prop2}) hold.

3. We show that the hypothesis of the first statement of the
theorem is true for the code $M(C,D)$ and a codeword ${x}^{1}$. By
the second statement of the theorem, the code $M(C,D)$ contains a
codeword ${x}^{1}$ with the incorrect inverse. Moreover the block
$\supp({x}^{1})$ is a subset $I((M(C,D)))=(I(C)\cup 0)\times
(I(D)\cup 0)\setminus (0,0)$, i.e.  the condition (\ref{C_C1}) is
satisfied. Finally, from (\ref{C_C3}) and (\ref{D_D1}) we have
that
$$\mu_{(r,0)}(M(C,D))=m<\frac{tm+t+m-1}{2},$$
$$\mu_{(0,s)}(M(C,D))=\mu_s(D)(t+1)+t<\frac{tm+t+m-1}{2},$$
$$\mu_{(r,s)}(M(C,D))<m<\frac{tm+t+m-1}{2},$$
so the condition (\ref{C_C2}) is fullfiled for $M(C,D)$.

\end{proof}

\begin{theorem}\label{Th2}

1. For $n=15$ there is a unique transitive nonpropelinear code.

 2.  For any $n\geq 15$ there is at least one transitive
nonpropelinear perfect code of length $n$.

\noindent   3. For any $n\geq 255$ there are at least 5 pairwise
inequivalent transitive nonpropelinear perfect codes of length
$n$.
\end{theorem}
\begin{proof}
1. See Proposition 1.

2. If $C$ is a unique transitive nonpropelinear perfect code of
length 15, then it fulfills the incorrect  inversion property for
$x$ such that $\supp(x)=\{2,3,4\}$, see Proposition \ref{TNP15}.
We show that the code $M(C,H)$ satisfies the condition of the
previous theorem for any Hamming code $H$ of length at least 1.

According to Table 1, $\supp(x)$ is a triple of a unique subsystem
of order 7 in $\mathrm{STS}(C)$ (here STS is treated as a binary
code).
 Since the coordinates of maximum order subsystems of $\mathrm{STS}$ are the complement of the supports of nonzero codewords in the dual code, see \cite{due}, we have that supp$(x)\subset I((\mathrm{STS}(C)))$.
 Since $C$ and $\mathrm{STS}(C)$ are both prefull rank codes (see Table \ref{TableTR15}) we have that
 $I(\mathrm{STS}(C))=I(C)$ and therefore (\ref{C_C1}) holds. Because there are no triples of $C$ in $\mathrm{Ker}(C)$, the condition (\ref{C_C2}) is also true.

3. The search shows that there are just 4 of 200 propelinear
perfect codes of length 15 with the condition that any such code
$D$ fulfills (\ref{D_D1}): $0<\mu_s(D)<7$ for any $s\in
\{1,\ldots,15\}$. These codes have numbers $51$, $694$, $724$,
$771$ in \cite{ost} (see also Table \ref{TableTR15}). If $D$ is
any such code then the code $M(M(C,D),H)$ is nonpropelinear.

 These four codes and the code $M(C,H')$ give five infinite series of nonpropelinear codes. From Table \ref{TableTR15} we have that the triple (rank, kernel dimension, parameter $\mu$) is a complete set of invariants determining inequivalence of the codes  with numbers $51$, $694$, $724$, $771$.
By (\ref{Mollard_Rank}) we see that the code $M(C,H')$ has
smaller rank then any code of the type $M(M(C,D),H)$ of the same
length. Moreover by (\ref{Mollard_Rank}), (\ref{Mollard_Ker}) and
Corollary \ref{mu_cor} the triple of invariants  remains to be
complete for the series of codes of the type $M(M(C,D),H)$.

\end{proof}

In some sense, the characteristic $\mu_i(C)$ of a perfect code $C$
could be seen as a measure of linearity of the coordinate $i$ in
the code. The transitive nonpropelinear perfect code of length 15
possesses the minimal linearity of coordinates, whereas the
Hamming code gives the maximum.

The essence of Theorem \ref{Table T1} could be described
informally: if we set $C$ to be the transitive nonpropelinear code
and choose $D$ to be a propelinear code
in a way that $\mu(D)$ is relatively "distant" to $\mu(C)$ in
order to preserve its nonpropelinearity for $M(C,D)$, but not very
"close" to $\mu(H)$ (where $H$ is a Hamming code) so that
$M(M(C,D),H)$ is nonpropelinear.

{\bf Acknowledgements} This work was initiated while Ivan
Mogilnykh was a visiting researcher at Combinatorics, Coding and
Security Group of Department of Information and Communications
Engineering in Autonomous University of Barcelona. He expresses
his sense of appreciation to the Head of the group Professor
J.~Rifa for the hospital and warm stay. The authors are grateful
to Josep Rifa and Quim Borges for stimulating discussions and
critical remarks on the text.




\bibliographystyle{elsarticle-num}



\end{document}